\documentclass{amsart}

\usepackage{amssymb}
\usepackage{float}
\usepackage{graphicx}
\usepackage[cmtip,all]{xy}

\newtheorem{theorem}{Theorem}[section]
\newtheorem{prop}{Proposition}[section]
\newtheorem{lemma}[theorem]{Lemma}
\newtheorem{cor}[theorem]{Corollary}

\theoremstyle{definition}
\newtheorem{definition}[theorem]{Definition}

\theoremstyle{remark}
\newtheorem{remark}[theorem]{Remark}

\numberwithin{equation}{section}

\newcommand{\abs}[1]{\left\lvert#1\right\rvert}
\newcommand{\ds}{\displaystyle}

\newcommand{\CE}{\mathcal{E}}
\newcommand{\CX}{\mathcal{X}}

\newcommand{\ff}{\mathfrak{f}}
\newcommand{\fm}{\mathfrak{m}}
\newcommand{\fp}{\mathfrak{p}}
\newcommand{\fP}{\mathfrak{P}}
\newcommand{\fq}{\mathfrak{q}}
\newcommand{\fG}{\mathfrak{G}}

\newcommand{\zpz}{\mathbb{Z}/p\mathbb{Z}}
\def\Gal{\mathop{\operatorfont Gal}\nolimits}

\begin{document}

\title[On lower bounds for the Ihara constants $A(2)$ and $A(3)$]
{On lower bounds for the Ihara constants $A(2)$ and $A(3)$}


\author{Iwan Duursma}
\address{Department of Mathematics \\
University of Illinois at Urbana-Champaign \\
273 Altgeld Hall, MC-382 \\
1409 W. Green Street \\
Urbana, Illinois 61801, USA}
\email{duursma@math.uiuc.edu}

\author{Kit-Ho Mak}
\address{Department of Mathematics \\
University of Illinois at Urbana-Champaign \\
273 Altgeld Hall, MC-382 \\
1409 W. Green Street \\
Urbana, Illinois 61801, USA}
\email{mak4@illinois.edu}


\begin{abstract}
Let $\mathcal{X}$ be a curve over $\mathbb{F}_q$ and let $N(\mathcal{X})$, $g(\mathcal{X})$ be its number of rational points and genus respectively. The Ihara constant $A(q)$ is defined by $A(q)=\limsup_{g(\mathcal{X})\rightarrow\infty}N(\mathcal{X})/g(\mathcal{X})$. In this paper, we employ a variant of Serre's class field tower method to obtain an improvement of the best known lower bounds on $A(2)$ and $A(3)$.
\end{abstract}

\maketitle

\section{Introduction}

Let $p$ be a prime and let $q=p^e$ be a prime power. Let $\CX$ be a projective, nonsingular, geometrically irreducible curve (hereafter referred to as \textit{curve}) of genus $g$. It is well-known that the Weil bound
\begin{equation*}
\#\CX(\mathbb{F}_q) \leq q+1+2g\sqrt{q}
\end{equation*}
is not sharp if $g$ is large compared to $q$. Put
\begin{equation*}
N_q(g):=\max{\#\CX(\mathbb{F}_q)},
\end{equation*}
where the maximum is taken over all curves $\CX/\mathbb{F}_q$ with genus $g$. The \textit{Ihara constant} is defined by
\begin{equation*}
A(q):=\limsup_{g\rightarrow\infty}\frac{N_g(q)}{g}.
\end{equation*}
This is a measure of the asymptotic behaviour of the number of rational points on curves over $\mathbb{F}_q$ when the genus becomes large. For any $q$, we have $A(q) \leq \sqrt{q}-1$ (see \cite{DrVl83}), and if $q$ is a square we have (see \cite{Iha81,TVZ82}) $A(q)=\sqrt{q}-1$. For a nonsquare $q$ we know much less. Serre proved that $A(q)\geq c\log{q}$ for some absolute constant $c$ by the class field tower method (see \cite{Ser85} and \cite[Chapter IX]{CaFr10}), and Temkine \cite{Tem01}, using the same method, improved this to $A(q^n)=(c'n\log{q})^2/(\log{n}+\log{q})$ with $c'$ an effectively computable constant. Variations of Serre's result are also obtained in \cite{NiXi99} and \cite{LiMa02}. Recently, using the recursive tower method, Garcia, Stichtenoth, Bassa and Beelen \cite{GSBB12} obtained very good lower bounds for $A(q^n)$ when $n>1$ is odd.

In this paper, we develop another variation of Serre's method, based on Kuhnt's work in his PhD thesis \cite{Kuh02}, for constructing class field towers. This yields the following lower bound for $A(q)$.
\begin{theorem} \label{thmtower}
Let $k$ be a function field of genus $g$ over $\mathbb{F}_q$ of characteristic $p$, where $q=p^e$. Let $S$ be a finite set of places, and for each $\fp \in S$, let $f_{\fp}$ denotes its degree and let $\nu_{\fp}$ be a positive integer. Form the conductor $\fm=\sum_{\fp\in S}\nu_{\fp}\fp$.
For a set $T$ of $t>0$ rational places disjoint from $S$ with $t \leq \sum_{\fp\in S}ef_{\fp}(\nu_{\fp}-1-[(\nu_{\fp}-1)/p])$,
let $K=k_{\fm}^T$ be the ray class field with conductor $\fm$ so that all places in $T$ split completely. If $t$ satisfies the inequality
\begin{multline}\label{ineqr-d}
(1+\sum_{\fp\in S}ef_{\fp}(\nu_{\fp}-1-[(\nu_{\fp}-1)/p])-t)^2 \\
- 2\sum_{\fp\in S}ef_{\fp}(\nu_{\fp}-1)(ef_{\fp}(\nu_{\fp}-1)+1) - 4\sum_{\fp\in S}ef_{\fp}(\nu_{\fp}-1-[(\nu_{\fp}-1)/p]) \geq 0,
\end{multline}
then the $(T,p)$-class field tower $L/K$ is infinite, and we have
\begin{equation*}
A(q) \geq \frac{t}{g-1+\frac{1}{2[K:k]}\sum_{\chi}\deg\ff_\chi},
\end{equation*}
where $\chi$ runs through the characters of $\Gal(K/k)$.
\end{theorem}

In general it may be difficult to determine exactly the degrees of the conductors for the characters of the group $\Gal(K/k)$, but we know that the conductors are bounded by $\fm=\sum_{\fp\in S}\nu_{\fp}\fp$ and their degrees by $\deg\fm=\sum_{\fp\in S}f_{\fp}\nu_{\fp}$. Therefore, we obtain the following corollary, which is slightly weaker than Theorem \ref{thmtower}, but has the advantage of being easier to use.
\begin{cor}\label{cortower}
Assumptions and settings as in Theorem \ref{thmtower}. We have
\begin{equation*}
A(q) \geq \frac{t}{g-1+\frac{1}{2}\sum_{\fp\in S}f_{\fp}\nu_{\fp}}.
\end{equation*}
\end{cor}

The purpose of this paper is twofold. First, since Kuhnt's method is not widely known, we will summarize it in this paper, and clarify some of his points. Second, using Theorem \ref{thmtower}, we further improve the lower bounds for $A(2)$ and $A(3)$.

\begin{theorem}\label{thmA2}
\begin{equation*}
A(2) \geq 0.316999\ldots.
\end{equation*}
\end{theorem}

\begin{theorem}\label{thmA3}
\begin{equation*}
A(3) \geq 0.492876\ldots.
\end{equation*}
\end{theorem}

Previous lower bounds for $A(2)$ appear in \cite{Ser85}, \cite{Sch92}, \cite{NiXi98}, \cite{XiYe07}, and lower bounds for $A(3)$ in
\cite{NiXi98}, \cite{Tem01}. Lower bounds for $A(3)$ using tamely ramified towers appear in \cite{AnMa02}, \cite[Section 4.2]{HaMa01}. Among these results,
the best lower bounds are
\[A(2)\geq 97/376 = 0.257979\ldots\]
by Xing and Yeo \cite{XiYe07}, and
\[A(3) \geq \frac{12}{25} = 0.48\]
by Atiken and Hajir \cite{HaMa01}.

In \cite{Kuh02}, Kuhnt obtained a better lower bound for $A(2)$, which says
\[A(2)\geq 39/129 = 0.302325\ldots.\]
For a survey about the recent developments on upper and lower bounds for $A(q)$, see \cite{Li07}.

We summarize the idea of the class field tower method in the next section. Two important inequalities will then be proved in Section \ref{secdp} and \ref{secr-d}. Finally, in Section \ref{secA23}, we will construct our towers and prove the theorems stated above.

\section{The class field tower method}\label{seccft}

Recall (see \cite[\S I.6]{Har77}) that there is an equivalence of categories between the category of curves (in the sense of the previous section) over $\mathbb{F}_q$ and the category of function fields of transcendence degree one (we will refer them hereafter as \textit{function fields}) over $\mathbb{F}_q$. Under this correspondence, the rational points of a curve correspond to the rational places (places of degree one) of its corresponding function field.

We briefly summarize the idea of Serre's class field tower method in this section. Suppose we have a curve $\CX$ over $\mathbb{F}_q$ of genus $g$, and let $K$ be its corresponding function field, called the \textit{ground field}. Let $l$ be a prime and $T$ a nonempty set of rational places in $K$. The $(T,l)$-Hilbert class field $H(K)$ of $K$ is the maximal unramified $l$-abelian extension of $K$ in which all places in $T$ split completely. We construct the $(T,l)$-class field tower using class field theory as follows. We construct recursively $K_1=H(K), K_2=H(K_1),\ldots$ and obtain the tower of fields
\[\ds K\subseteq K_1\subseteq K_2\subseteq\ldots.\]
If the tower is infinite, then using the Hurwitz genus formula, we get a lower bound for $A(q)$ with
\begin{equation} \label{eqnT/g}
A(q) \geq \frac{\abs{T}}{g-1}.
\end{equation}
See \cite[Section 2.7]{NiXi01} for more details.


One of the essential tools of the class field tower method is the Golod-Shafarevich theorem \cite{GoSa64,GaNe70}.
\begin{theorem}[Golod-Shafarevich]
Let $p$ be a prime and let $G$ be a nontrivial finite $p$-group. We have
\begin{equation*}
r_p(G)>\frac{d_p(G)^2}{4}.
\end{equation*}
\end{theorem}

If the $(T,\ell)$-Hilbert class field tower of $K$ stabilizes at a finite extension $L/K$ then the relation rank and the generator rank of $G'=\Gal(L/K)$ satisfy
bounds of the form $d(G') \geq A'$ and $r(G')-d(G') \leq B'$. When $A'$ and $B'$ are such that $B' \leq A'^2/4-A'$ the possibility that the tower stabilizes can be excluded.
In \cite{Kuh02}, Kuhnt considers the tower of extensions $L/K/k$, where $L/K$ is the usual class field tower of Serre and $K/k$ is a finite Galois $p$-extension
with controlled wild ramification. He then considers the Galois group $G=\Gal(L/k)$ instead of the usually considered $G'=\Gal(L/K)$, and show that the same type of inequalities hold for $G$. i.e. $d(G) \geq A$ and $r(G)-d(G) \leq B$ for some $A$ and $B$. Again if $B \leq A^2/4-A$, then $G$ is infinite and so is the class field tower.

Theorem \ref{thm626} in Section~\ref{secdp} gives a lower bound for the generator rank that applies to $\Gal(L/K)$ and with a minor modification to $\Gal(L/k)$. Theorem \ref{thm810} in Section \ref{secr-d} gives an upper bound for the difference of the relation rank and the generator rank for the group $\Gal(L/k)$. The theorem is due to Kuhnt \cite{Kuh02}. In Section~\ref{secA23} we give a general method of constructing towers, and state the criteria for the towers to be infinite. This allow us to improve the known lower bounds for $A(2)$ and $A(3)$.


\section{Lower Bound for the Generator Rank} \label{secdp}

\subsection{Ramification of bounded depth}

Let $K$ be a global function field of characteristic $p$ with exact constant field $\mathbb{F}_q$. Let $S$ be a set of primes in $K$, and $\nu:S\rightarrow[0,\infty]$ be a map sending $\fp$ to $\nu_{\fp}$. We extend $\nu$ to all primes in $K$ by setting $\nu_{\fp}=0$ for all $\fp\notin S$. The first concept we need is ramification of bounded depth. We will be contented with outlining only the necessary background, for more details see \cite{HaMa02}.

\begin{definition}[See Section 3 of \cite{HaMa02}]
\text{}
\begin{enumerate}
\item Let $L/K$ be a Galois extension of global fields with Galois group $G$. We say that $L/K$ has ramification of depth at most $n$ at a prime $\fp$ in $K$ if the ramification groups $G_{\fP}^n$ in upper numbering (normalized as in \cite{Ser79}) are trivial for all primes $\fP$ in $L$ above $\fp$.
\item Let $\nu:S\rightarrow[0,\infty]$ be a map. We say that the ramification depth of $L/K$ is bounded by $\nu$ if $L/K$ has ramification of depth at most $\nu_{\fp}$ at any prime $\fp$.
\item Let $K_{S,\nu}$ be the maximal $p$-extension of $K$ unramified outside $S$ and with the property that the ramification depth of $K_{S,\nu}/K$ at $\fp$ is at most $\nu_{\fp}$ for any prime $\fp\in S$. Let $G_{S,\nu}=\Gal(K_{S,\nu}/K)$. This is a $p$-group.
\end{enumerate}
\end{definition}

Suppose $L/K$ is a finite extension contained in $K_{S,\nu}$, and $S_L$ is the set of all primes in $L$ lying above $S$. We lift the map $\nu$ in $S$ to a map $\nu_L$ in $S_L$ by setting
\begin{equation} \label{eqnnuL}
\nu_{L,\fP}=\psi_{\fP/\fp}(\nu_{\fp}),
\end{equation}
where $\psi_{\fP/\fp}$ is given by the equation $G^s_{\fP}=G_{\psi_{\fP/\fp}(s),\fP}$ relating the upper numberings and lower numberings of the ramification groups. This definition allows us to describe the extension $K_{S,\nu}/K$ as a tower of abelian extensions. Set $K_1=K$, $S_1=S$. We define $K_{n+1}$ to be the maximal abelian extension of $K_n$ contained in $K_{S_n,\nu_n}$, and $S_{n+1}$ the set of primes in $K_{n+1}$ lying above $S_n$, and $\nu_{n+1}$ the extension of $\nu$ from $S_n$ to $S_{n+1}$. Let $K_{\infty}$ be the union of all $K_n$.
\begin{prop}[Theorem 3.5 of \cite{HaMa02}]
Let $K$, $S$, $\nu$ be as above. Then $K_{\infty}=K_{S,\nu}$.
\end{prop}

\subsection{The generator rank}

Following the notations from \cite{HaMa02}, we set
\begin{align*}
U_{\fp}^{(n)} &= \{ x\in K_{\fp}|v_{\fp}(x-1)\geq n \} \text{~(the $n$-th higher unit group in $K_{\fp}$)}, \\
\Delta &= \{ x\in K^{\ast}|(x) \text{~is a $p$-th power in the group of fractional ideals of $K$} \}, \\
\Delta_S &= \{ x\in\Delta|x\in K_{\fp}^{\ast p} \,\,\forall\fp\in S \}/K^{\ast p}, \\
\Delta_{S,\nu} &= \{ x\in\Delta|x\in K^{\ast p}_{\fp}U_{\fp}^{(\nu_{\fp})} \,\,\forall\fp\in S \}/K^{\ast p}.
\end{align*}
Then the generator rank $d_{S,\nu}$ of $G_{S,\nu}$ is given by the following proposition.
\begin{prop}[Theorem 6.26 of \cite{Kuh02}] \label{prop626}
\begin{equation*}
d_{S,\nu}=1+d_p(\Delta_{S,\nu})+\sum_{\fp\in S}d_p\left(\frac{U_{\fp}^{(1)}}{U_{\fp}^{(\nu_{\fp})}}\right).
\end{equation*}
\end{prop}
\begin{proof}
This proposition is the function field analogue of \cite[Theorem 3.7]{HaMa02}, and the proof follows the same line as in that theorem. We outline the proof here. Let $I_K$ be the group of ideles of $K$, and let $\mathcal{U}_{S,\nu}=\prod_{\fp\notin S}U_{\fp}\prod_{\fp\in S}U_{\fp}^{(\nu_{\fp})}$. In particular, if $S=\varnothing$, we write $\mathcal{U}_{\varnothing}=\prod_{\fp}U_{\fp}$. Consider the following exact sequence
\begin{equation*}
1 \rightarrow \Delta_{S,\nu} \rightarrow \Delta/K^{\ast p} \rightarrow \frac{\mathcal{U}_{\varnothing}}{\mathcal{U}^p_{\varnothing}\mathcal{U}_{S,\nu}} \rightarrow \frac{I_K}{\mathcal{U}_{S,\nu}I^p_K} \rightarrow \frac{I_K}{K^{\ast}\mathcal{U}_{\varnothing}I^p_K} \rightarrow 1
\end{equation*}
in class field theory. Note that $I_K/K^{\ast}\mathcal{U}_{\varnothing}I^p_K=Cl_K/Cl_K^p$, and $d_p(\Delta_{\varnothing})=d_p(Cl_K)$ since $d_p(\mathbb{F}_q^{\ast})=0$. Let $\fm=\sum_{\fp\in S}\nu_{\fp}\fp$ be the conductor corresponding to $\nu$. Taking account of the constant fields in $K_{S,\nu}$ and $Cl_K$, we obtain $d_p(G_{S,\nu})=d_p(Cl^{\fm}(K))+1$. Putting all these together gives the proposition.
\end{proof}

To calculate the generator rank above, we need the following.
\begin{prop} \label{prop627}
Let $K$ be a global function field of characteristic $p$, with full constant field $\mathbb{F}_q$, where $q=p^e$. Let $\fp$ be a prime in $K$ of degree $f$, then
\begin{equation*}
d_p\left(\frac{U_{\fp}^{(1)}}{U_{\fp}^{(\nu_{\fp})}}\right)=f\cdot e\cdot (\nu_{\fp}-1-[(\nu_{\fp}-1)/p]),
\end{equation*}
where $[\cdot]$ denotes the integer part.
\end{prop}
\begin{proof}
This is \cite[Lemma 4.2.5 (i)]{NiXi01}.
\end{proof}

Combining Proposition \ref{prop626} and \ref{prop627}, we get the following theorem.
\begin{theorem} \label{thm626}
Let $K$ be a global function field of characteristic $p$, with full constant field $\mathbb{F}_q$, where $q=p^e$. For a prime $\fp$ in $K$ let $f_{\fp}$ be its degree. The generator rank $d_{S,\nu}$ of $G_{S,\nu}=\Gal(K_{S,\nu}/K)$ satisfies
\begin{equation*}
d_{S,\nu} = 1+d_p\Delta_{S,\nu}+\sum_{\fp\in S}ef_{\fp}(\nu_{\fp}-1-[(\nu_{\fp}-1)/p]).
\end{equation*}
\end{theorem}

\section{Difference between the generator rank and the relation rank} \label{secr-d}

The following theorem about the estimation of the difference between the generator rank and the relation rank is due to Kuhnt \cite{Kuh02}.

\begin{theorem}[Theorem 8.10 of \cite{Kuh02}] \label{thm810}
Let $L/K/k$ be a tower of Galois $p$-extensions of global function fields over a finite field. The extension $K/k$ is finite, unramified outside a set $S$ of primes of $k$ and $L/K$ is unramified. Let $T_L$ be a nonempty, finite set of primes in $L$ and let $T_k=T_L\cap k$. Let $G=\Gal(L/k)$ and let
\begin{equation*}
\delta = r_p(G)-d_p(G)-\sum_{\fp\in S}(r_p(G_{\fp})-d_p(G_{\fp}))-(\abs{T_k}-1)+d_p(\Delta_{S}).
\end{equation*}
If $\delta>0$, then there exists an unramified Galois $p$-extension $\tilde{L}/L$, 
in which $T_L$ splits completely with $d_p(\Gal(\tilde{L}/L))\geq\delta$ and with the same constant field as $L$. In other words, if $L$ is the maximal unramified $p$-extension of $K$ (and $S, T_L, T_k$ as above), then
\begin{equation*}
r_p(G)-d_p(G) \leq \sum_{\fp\in S}(r_p(G_{\fp})-d_p(G_{\fp}))+(\abs{T_k}-1)-d_p(\Delta_{S}).
\end{equation*}
\end{theorem}

We first recall some preliminary results.

\subsection{The embedding problem}

The embedding problem is the induction step in the construction of field extensions with prescribed Galois group. We will outline the necessary background in this subsection. For details, see \cite{Neu73,NSW08}.

\begin{definition}
Let $\fG$ be a profinite group. An \textit{embedding problem} $\CE(\fG)$ is a diagram
\begin{equation*}
\xymatrix{
&&& \fG \ar@{->}[d]^{\phi} \\
\epsilon:0\ar@{->}[r] & A\ar@{->}[r] & E\ar@{->}[r]^-j & G\ar@{->}[r] & 1
}
\end{equation*}
where $A$, $E$, $G$ are finite groups, $A$ is abelian, $\epsilon$ is a short exact sequence and $\phi$ is surjective. A solution to the problem is a continuous homomorphism $\chi:\fG\rightarrow E$ making the above diagram commutative.
\begin{equation*}
\xymatrix{
&&& \fG \ar@{->}[d]^{\phi} \ar@{.>}[dl]_{\chi} \\
\epsilon:0\ar@{->}[r] & A\ar@{->}[r] & E\ar@{->}[r]^-j & G\ar@{->}[r] & 1
}
\end{equation*}
The solution $\chi$ is called proper if it is surjective. The group $A$ is called the \textit{kernel} of the embedding problem $\CE(\fG)$.
\end{definition}

The following proposition is useful to determine if an embedding problem has a solution. Let $inf_{\fG}^G:H^2(G,A)\rightarrow H^2(\fG,A)$ be the inflation map induced from $\phi:\fG\rightarrow G$.
\begin{prop}[Hoechsmann]\label{propHoech}
The embedding problem $\CE(\fG)$ has a solution if and only if $inf_{\fG}^G(\epsilon)=1$.
\end{prop}
\begin{proof}
See \cite[Prop. 3.5.9]{NSW08}. The original proof by Hoechsmann can be found in \cite{Hoe68}.
\end{proof}

Now let $\fG_K$ be the absolute Galois group of $K$, let $L/K$ be a finite Galois extension, and let $G=\Gal(L/K)$ be its Galois group. Let $S$ be a set of primes in $K$ and $S_L$ the set of primes in $L$ lying above $S$. The global embedding problem $(L/K,\epsilon,S_L)$ is defined as follows.

\begin{definition}
Let $L/K$ be a finite extension, and $G=\Gal(L/K)$ its Galois group. Let
\begin{equation*}
\xymatrix@1{
\epsilon:0\ar@{->}[r] & A\ar@{->}[r] & E\ar@{->}[r]^-j & G\ar@{->}[r] & 1
}
\end{equation*}
be an exact sequence, where $A$, $E$ are finite groups with $A$ abelian. A \textit{global embedding problem} $(L/K,\epsilon,S_L)$ is a diagram
\begin{equation*}
\xymatrix{
&&& \fG_K \ar@{->}[d]^{\phi} \\
\epsilon:0\ar@{->}[r] & A\ar@{->}[r] & E\ar@{->}[r]^-j & G\ar@{->}[r] & 1
}
\end{equation*}
where $\phi$ is the restriction map.

A \textit{solution} of the embedding problem is a continuous homomorphism $\chi:\fG_K\rightarrow E$ making the diagram
\begin{equation*}
\xymatrix{
&&& \fG_K \ar@{->}[d]^{\phi} \ar@{.>}[dl]_{\chi} \\
\epsilon:0\ar@{->}[r] & A\ar@{->}[r] & E\ar@{->}[r]^-j & G\ar@{->}[r] & 1
}
\end{equation*}
commutative, and such that $M/L$ is unramified outside $S_L$, where $M$ is the fixed field of the kernel of $\chi$. The solution $\chi$ is called \textit{proper} if it is surjective.
\end{definition}

Note that in the above setting, a proper solution of the embedding problem $(L/K,\epsilon,S_L)$ corresponds to a Galois extension $M/K$, with $\Gal(M/L)=A$ and is unramified outside $S_L$. By abuse of notation we also say that $M$ is a solution of $(L/K,\epsilon,S_L)$. One important fact we need is that if $L/K$ is a $p$-extension, the embedding problem with kernel $\zpz$ is solvable. More precisely, we have the following.

\begin{prop}[Theorem 8.1 of \cite{Kuh02}] \label{prop81}
Let $L/K$ be a Galois $p$-extension, and let $S_L$ be a finite, non-empty set of primes in $L$ containing the ramified primes of $L/K$. Let
\begin{equation*}
\xymatrix@1{
\epsilon:0\ar@{->}[r] & \zpz\ar@{->}[r] & E\ar@{->}[r] & G(L/K) \ar@{->}[r] & 1
}
\end{equation*}
be a non-split extension (meaning that $\epsilon$ is a non-split exact sequence). If $S_L\cap K=S_K$, then the global embedding problem $(L/K,\epsilon,S_L)$ has a proper solution $M$ with the same constant field as $L$.
\end{prop}
\begin{proof}
Let $\fG_{S_{K,p}}$ be the Galois group of the maximal pro-$p$-extension of $K$ unramified outside $S_K$. Then we have the diagram
\begin{equation*}
\xymatrix{
&&& \fG_{S_{K,p}} \ar@{->}[d]^{\phi} \ar@{.>}[dl]_{\chi} \\
\epsilon:0\ar@{->}[r] & \zpz \ar@{->}[r] & E \ar@{->}[r]^-j & G(L/K) \ar@{->}[r] & 1
}
\end{equation*}
where $\phi$ is the restriction map. Since $H^2(\fG_{S_{K,p}},\zpz)$ is trivial (see \cite[Chapter VIII]{NSW08}), we must have $inf^{G(L/K)}_{\fG_{S_{K,p}}}(\epsilon)=1$. The proposition now follows from Proposition \ref{propHoech}.
\end{proof}

\subsection{Independence of solutions to the global embedding problems}

Let $L/K$ be a Galois $p$-extension with Galois group $G=\Gal(L/K)$. Fix a finite, non-empty set $S_L$ of places in $L$ containing the ramified places of $L/K$. For each $\epsilon\in H^2(G,\zpz)$, the corresponding global embedding problem $(L/K,\epsilon,S_L)$ has a proper solution by Proposition \ref{prop81}. The main result in this section is the following.

\begin{prop}\label{prop318}
Let $\epsilon_1, \ldots, \epsilon_n$ be $n$ linearly independent elements in $H^2(G,\zpz)$, and let $M_1, \ldots, M_n$ be proper solutions to the global embedding problems corresponding to $\epsilon_i$. Let $M=M_1\ldots M_n$ be the compositum of the solutions. Then $d_p(\Gal(M/L))=n$.
\end{prop}

To prove Proposition \ref{prop318}, we start with some observations. Let $E(G,A)$ be the set of equivalence classes of central extensions of a group $G$ by an abelian group $A$. It is well-known that there is a bijection
\begin{equation}\label{eqnHEcor}
\xymatrix@1{
E(G,A) \ar@{<->}[r]^{\sim} & H^2(G,A)
}
\end{equation}
functorial at the slot of $A$. See \cite[Sec. 6.6]{Wei94} for details.

\begin{lemma}\label{lem314}
Let $G$ be a group, and let $A_1, \ldots, A_n$ be abelian groups. For each $1\leq i \leq n$, let $\epsilon_i\in H^2(G,A_i)$ be the element corresponding to the extension
\begin{equation*}
\xymatrix@1{
\epsilon_i:0\ar@{->}[r] & A_i \ar@{->}[r] & E_i\ar@{->}[r] & G \ar@{->}[r] & 1
}
\end{equation*}
under the bijection \eqref{eqnHEcor}. Then the element $(\epsilon_1,\ldots,\epsilon_n)\in H^2(G,\oplus_{i=1}^n A_i)=\oplus_{i=1}^n H^2(G,A_i)$ corresponds to the extension
\begin{equation*}
\xymatrix@1{
\epsilon_i:0\ar@{->}[r] & \oplus_{i=1}^n A_i \ar@{->}[r] & E_1\times_G E_2 \times_G \ldots \times_G E_n \ar@{->}[r] & G \ar@{->}[r] & 1.
}
\end{equation*}
\end{lemma}
\begin{proof}
For $n=2$ the lemma follows from a detailed tracing of the bijection in \eqref{eqnHEcor}. The case for general $n$ then follows by induction.
\end{proof}

\begin{lemma}\label{lem315}
Let $L, K, S_L$ and $G$ be as in the first paragraph of this subsection. For $\epsilon_1, \ldots, \epsilon_n\in H^2(G,\zpz)$, let $M_1,\ldots,M_n$ be solutions to the global embedding problems $(L/K,\epsilon_i, S_L)$. If we have $M_i\cap\prod_{j\neq i}M_j = L$ for all $i$, then $M$ is a solution to the global embedding problem $(L/K, (\epsilon_1,\ldots,\epsilon_n) , S_L)$.
\end{lemma}
\begin{proof}
If $M_i\cap\prod_{j\neq i}M_j = L$ for all $i$, then we know from Galois theory that
\begin{equation*}
\Gal(M/K) = \Gal(M_1/K) \times_G \ldots \times_G \Gal(M_n/K).
\end{equation*}
The lemma now follows from Lemma \ref{lem314} directly.
\end{proof}

\begin{proof}[Proof of Proposition \ref{prop318}]
Let $\tilde{M}_i = M_1M_2\ldots M_i$, thus $M=\tilde{M}_n$. Clearly $d_p(\Gal(M/L))\leq n$. Suppose that $d_p(\Gal(M/K))<n$. Let $j$ be the smallest integer $j$ such that
\begin{equation}\label{eqncond318}
d_p(\Gal(\tilde{M}_j)/L)=d_p(\Gal(\tilde{M}_{j+1})/L)=j.
\end{equation}
In this case we have $M_{j+1}\subseteq \tilde{M}_j$ since each $\Gal(M_i/L)\cong \zpz$. The extensions $M_i/L$ are Galois of degree $p$, so they are Artin-Schreier extensions \cite[Theorem VI.6.4]{Lan02}. Hence $\Gal(\tilde{M}_j/L)\cong(\zpz)^j$. The degree $p$ subfields of $\tilde{M}_j$ are the fixed fields of the kernels of the non-trivial homomorphisms $(\zpz)^j\rightarrow\zpz$, which are linear maps. These maps induce linear maps $H^2(G,(\zpz)^j)\rightarrow H^2(G,\zpz)$ on the cohomology groups. In particular, the homomorphism $\phi:\Gal(\tilde{M}_j/L)\rightarrow \Gal(M_{j+1}/L)$ that realizes the extension $\tilde{M}_j/M_{j+1}/L$ corresponds to the map of extensions
\begin{equation*}
\xymatrix{
\epsilon:1\ar@{->}[r] & \Gal(\tilde{M}_j/L) \ar@{->}[r] \ar@{->>}[d]^{\phi} & \Gal(\tilde{M}_j/K) \ar@{->}[r] \ar@{->>}[d]^{\tilde{\phi}} & G \ar@{->}[r] \ar@{=}[d] & 1 \\
\phi^{\ast}\epsilon:1\ar@{->}[r] & \Gal(M_{j+1}/L) \ar@{->}[r] & \Gal(M_{j+1}/K) \ar@{->}[r] & G \ar@{->}[r] & 1.
}
\end{equation*}

By \eqref{eqncond318}, the fields $M_1, \ldots, M_j$ satisfy $M_i\cap\prod_{l\neq i}M_l = L$ for all $i$. Therefore, $\tilde{M}_j$ corresponds to the element $\epsilon=(\epsilon_1,\ldots,\epsilon_j)$ by Lemma \ref{lem315}, and the $p$ subextension $M_{j+1}$ of $\tilde{M}_j$ corresponds to $\phi^{\ast}(\epsilon)$, which is a linear combination of $\epsilon_1,\ldots,\epsilon_j$ as $\phi$ is linear. On the other hand, we know that the $p$-extension $M_{j+1}$ corresponds to $\epsilon_{j+1}$. Thus $\epsilon_{j+1}$ is a linear combination of $\epsilon_1,\ldots,\epsilon_j$. This is a contradiction.
\end{proof}

\subsection{Unramified cohomology}

In this subsection we describe the necessary background on unramified cohomology. For details, see \cite{NSW08,Ser02}.

\begin{definition}
Let $L/K$ be a Galois extension of global function fields with Galois group $G$. For a prime $\fp$ in $K$, and a $G$-module $A$, define 
\begin{equation*}
\xymatrix@1{
H^i_{nr}(G_{\fp},A)=\text{im}(H^i(G_{\fp}/I_{\fp},A^{I_{\fp}})\ar@{->}[r]^-{inf} & H^i(G_{\fp},A)),
}
\end{equation*}
where $inf$ is the inflation map.
\end{definition}

We are interested in the unramified cohomology $H^2_{nr}(G_{\fp},\zpz)$. If $L/K$ is a $p$-extension, and $\fp$ is unramified in $L/K$, then
\begin{equation*}
H^2_{nr}(G_{\fp},\zpz)=H^2(G_{\fp},\zpz)\cong\zpz
\end{equation*}
if $G_{\fp}\neq 1$. Let $L/K$ be a $p$-extension of global fields with Galois group $G$ unramified outside a set $S$ of primes of $K$. Define $S_{nr}=\{ \fp\in S|d_p(I_{\fp}G'_{\fp}/G'_{\fp})=d_p(G_{\fp})-1 \}$, where $G'$ is the commutator subgroup of $G$.
\begin{remark}\label{rem632}
Suppose $L/K$ is as above, and $\fp$ is a prime in $S$. Let $L_{\fp}^{ab}/K_{\fp}$ be the maximal abelian subextension of $L_{\fp}/K_{\fp}$. Then the set $S_{nr}$ consists of exactly those primes in $S$ such that $d_p(G_{\fp})=d_p(G_{\fp}^{ab})=d_p(I_{\fp}^{ab})+1$. The other primes $\fp\in S\backslash S_{nr}$ satisfy $d_p(G_{\fp})=d_p(I_{\fp}^{ab})$.
\end{remark}
If $\fp$ is ramified, the unramified cohomology is given by the following proposition.
\begin{prop}[Theorem 6.33 of \cite{Kuh02}] \label{prop633} 
Let $L/K$ be a $p$-extension with Galois group $G$ unramified outside a set $S$ of primes in $K$. Let $\fp\in S$, then
\begin{equation*}
H^2_{nr}(G_{\fp},\zpz)=
\begin{cases}
\zpz &,~\fp\in S_{nr}, \\
1 &, \text{~otherwise}.
\end{cases}
\end{equation*}
\end{prop}
\begin{proof}
Define $f$ by $G_{\fp}/I_{\fp}\cong\mathbb{Z}/p^f\mathbb{Z}$. From the exact sequence
\begin{equation*}
1\rightarrow I_{\fp}\rightarrow G_{\fp}\rightarrow G_{\fp}/I_{\fp}\rightarrow 1,
\end{equation*}
we obtain by the Lyndon-Hochschild-Serre spectral sequence (see for example \cite{Wei94}) the following exact sequence.
\begin{equation*}
\xymatrix{
1 \ar@{->}[r] & H^1(G_{\fp}/I_{\fp},\zpz) \ar@{->}[r]^-{inf} & H^1(G_{\fp},\zpz) \ar@{->}[r]^-{res_1} & H^1(I_{\fp},\zpz)^{G_{\fp}/I_{\fp}} \\
& \ar@{->}[r] & H^2(G_{\fp}/I_{\fp},\zpz) \ar@{->}[r]^-{inf_2} & H^2(G_{\fp},\zpz).
}
\end{equation*}
Since $H^2(\mathbb{Z}/p^f\mathbb{Z},\zpz)\cong\zpz$, we have $H^2_{nr}(G_{\fp},\zpz)\cong \zpz$ if and only if $inf_2$ is injective, and $H^2_{nr}(G_{\fp},\zpz)=1$ otherwise. By exactness, $inf_2$ is injective if and only if $res_1$ is surjective. Looking at the $p$-ranks we see that this is equivalent to $d_p(H^1(I_{\fp},\zpz))=d_{\fp}(G_{\fp})-1$. By a little computation we see that $d_p(H^1(I_{\fp},\zpz))=d_p(I_{\fp}G'_{\fp}/G'_{\fp})$. This finishes the proof of the proposition.
\end{proof}

\subsection{Class field theory of wildly ramified extensions}

The following lemmas on class field theory of wildly ramified extensions are used in the proof of Theorem \ref{thm810}.

\begin{lemma}[Lemma 5.13 of \cite{Kuh02}] \label{lem513}
Let $L/K$ be a Galois extension of global fields, and let $K'/K$ be a Galois extension linearly disjoint from $L$. Let $\fp$ be a prime of $K$, and let $\fP$, $\fP'$, $\fp'$ be compatible prolongations of $\fp$ to $L$, $LK'$ and $K'$ respectively. If $K'/K$ is ramified of depth at most $t$ at $\fp$, then $LK'/L$ is ramified of depth at most $s=\psi_{L/K,\fp}(t)$ at $\fP$, where $\psi$ is given by the equation $G^s_{\fP}=G_{\psi_{\fP/\fp}(s),\fP}$ relating the upper and lower numberings of the ramification groups.
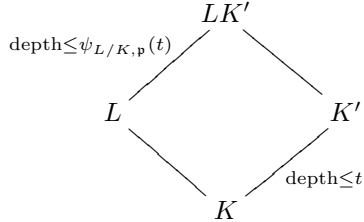
\begin{equation*}
\xymatrix{
& LK' \ar@{-}[dl]_{\textup{depth} \leq \psi_{L/K,\fp}(t)} \ar@{-}[dr] \\
L \ar@{-}[dr] & & K' \ar@{-}[dl]^{\textup{depth} \leq t} \\
& K
}
\end{equation*}
\end{lemma}
\begin{proof}
The proof is by tracing the definitions of the ramification groups.
\end{proof}

\begin{lemma}[Theorem 5.18 of \cite{Kuh02}] \label{lem518}
Let $L/K$ be an \textit{elementary abelian} $p$-extension of global fields of characteristic $p$ with Galois group $G$, and let $\fp$ be a prime of $K$. Then the upper ramification jumps at $\fp$ are prime to $p$.
\end{lemma}
\begin{proof}
The statement is well-defined by the Hasse-Arf theorem (see \cite[p.76]{Ser79}). The lemma is then proved by induction on the number of jumps.
\end{proof}


\begin{lemma}[Corollary 5.23 of \cite{Kuh02}] \label{lem523}
Let $L/K$ be a $p$-extension of global function fields of characteristic $p$, with exact constant field $\mathbb{F}_{p^e}$. Let $\fP$ be a prime of $L$ ramified in $L/K$ and let $\fp=\fP\cap K$. Assume $L/K$ is ramified of depth at most $\nu_{\fp}$ at $\fp$, and let $I_{\fp}$ be the inertia group at $\fp$. If $L_{\fP}/K_{\fp}$ is abelian, then
\begin{equation*}
d_p(I_{\fp}) \leq e\cdot\deg\fp\cdot(\nu_{\fp}-1-[(\nu_{\fp}-1)/p]).
\end{equation*}
\end{lemma}
\begin{proof}
Since we are only interested in $p$-ranks, it suffices to prove the lemma when $I_{\fp}$ is elementary abelian. By the Hasse-Arf theorem, all jumps in the upper filtration are integers. By Lemma \ref{lem518}, the jumps are prime to $p$, and by \cite[Prop. IV.7]{Ser79}, the $p$-rank is decreased by at most $e\cdot\deg{\fp}$ for each jump. The lemma now follows from a simple counting argument.
\end{proof}

\subsection{Proof of Theorem \ref{thm810}}

\begin{proof}
\textbf{(Step 1: Find proper solutions of local embedding problems.)} Recall that $G=\Gal(L/k)$. Define
\begin{equation*}
H^2_{nr}(G,\zpz)=\{ \epsilon\in H^2(G,\zpz) |~ res^G_{G_\fp}(\epsilon)\in H^2_{nr}(G_\fp,\zpz) \,\,\forall\fp\in S \}.
\end{equation*}
By Proposition \ref{prop81}, for each nonzero $\epsilon\in H^2_{nr}(G,\zpz)$ we can find a proper solution $M_{\epsilon}$ of the embedding problem $(L/K,\epsilon,S_L)$. Choose a basis $\epsilon_1,\ldots,\epsilon_n$ of $H^2_{nr}(G,\zpz)$ and let $N$ be the compositum of all the $M_{\epsilon_i}$.  Note that $N$ is Galois over $k$ since each $M_{\epsilon_i}$ is, and $\Gal(N/L)$ is central in $\Gal(N/k)$ because the $G$-action on $\zpz$ (which is the kernel of each embedding problem) is trivial. By Proposition \ref{prop318}, we have
\begin{equation}\label{dpNL}
d_p \Gal(N/L)=d_p H^2_{nr}(G,\zpz).
\end{equation}
To calculate the $p$-rank of $H^2_{nr}(G,\zpz)$, consider the restriction map
\begin{equation*}
\xymatrix@1{
H^2(G,\zpz) \ar@{->}[r]^-{\prod res} & \prod_{\fp\in S}\left(H^2_{nr}(G_{\fp},\zpz)\oplus H^2_{nr}(G_{\fp},\zpz)^{comp}\right),
}
\end{equation*}
where $H^2_{nr}(G_{\fp},\zpz)^{comp}$ is a complement (as a $\mathbb{F}_p$-vector space) of $H^2_{nr}(G_{\fp},\zpz)$ in $H^2(G_{\fp},\zpz)$. By Proposition \ref{prop633}, we have
\begin{equation*}
d_p H^2_{nr}(G_{\fp},\zpz)^{comp}=
\begin{cases}
d_p H^2_{nr}(G_{\fp},\zpz)-1 &, \fp\in S_{nr}, \\
d_p H^2_{nr}(G_{\fp},\zpz) &, \fp\notin S_{nr}.
\end{cases}
\end{equation*}
Therefore, we have
\begin{align*}
& d_p H^2_{nr}(G,\zpz) \\
\geq~ &d_p H^2(G,\zpz)-\sum_{\fp\in S\backslash S_{nr}} d_p H^2(G_{\fp},\zpz)-\sum_{\fp\in S_{nr}} \left(d_p H^2(G_{\fp},\zpz)-1\right) \\
=~ &r_p(G)-\sum_{\fp\in S\backslash S_{nr}} r_p(G_{\fp})-\sum_{\fp\in S_{nr}} (r_p(G_{\fp})-1).
\end{align*}
Combining this with \eqref{dpNL}, we have
\begin{equation}\label{dpstep1}
d_p \Gal(N/L) \geq r_p(G)-\sum_{\fp\in S\backslash S_{nr}} r_p(G_{\fp})-\sum_{\fp\in S_{nr}} (r_p(G_{\fp})-1).
\end{equation}

\textbf{(Step 2: Remove the ramification exceeding the ramification depth of $L/k$.)} Here, by ``removing'' the ramification of an extension $N/L$ exceeding certain ramification depth $\nu$, we mean to replace $N$ by a subextension $N_1$ of $N$ over $L$ so that the ramification depth of $N_1/L$ is bounded by $\nu$.

For the extension $L/k$ and $\fp$ a prime in $k$, let $\nu_{\fp}$ be the ramification depth at $\fp$. Let $S_L$ be the primes in $L$ lying above $S$, and $\nu_L$ be the lift of $\nu$ in $S_L$ using \eqref{eqnnuL}. If the ramification of $N$ obtained in Step $1$ over $L$ is still bounded by $\nu_L$, go to step 3.

If not, let $\fq$ be a prime in $L$ so that $N/L$ has ramification depth at $\fq$ exceeding $\nu_{L,\fq}$. For any prime $\fp$ in $S$, let $L_{\fp}^{nr}$ be the maximal unramified subextension of $L_{\fp}/k_{\fp}$. Since the extension $N/L$ obtained in Step $1$ comes from $H^2_{nr}(G,\zpz)$, the ramification index of any place $\fp\in S$ over $L$ is at most $p$. Hence, we have $N_{\fp}/L_{\fp}=N_{\fp}'L_{\fp}/L_{\fp}$ for some $N_{\fp}'$ elementary abelian over $L_{\fp}^{nr}$. Let $n_{\fp}$ be the ramification depth of $N_{\fp}'/L_{\fp}^{nr}$. Then $n_{\fp}-1$ is the highest ramification break of the extension. Since $L_{\fp}^{nr}/k_{\fp}$ is unramified, $n_{\fp}-1$ is also the highest ramification break of $N_{\fp}'/k_{\fp}$.

Let $M_1$ be the compositum of $N$ with all elementary abelian $p$-extensions of $K$ ramified at the same primes as $N/L$ of depth at most $n_{\fp}$ for all $\fp\neq\fq$ and of exact depth $n_{\fq}$ at $\fq$. By Lemma \ref{lem513}, this does not change the ramification depth of any local extensions. In particular, we have $M_{1,\fp}/L_{\fp}=M'_{1,\fp}L_{\fp}/L_{\fp}$, with $M'_{1,\fp}/L_{\fp}^{nr}$ elementary abelian. The tower of the local fields are as shown in Figure \ref{figstep2}.

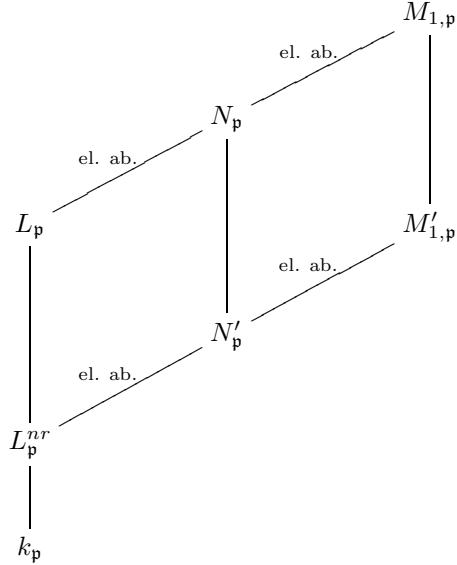
\begin{figure}[H]
\begin{equation*}
\xymatrix{
&&&& M_{1,\fp} \ar@{-}[dll]_{\text{el. ab.}} \ar@{-}[dd] \\
&& N_{\fp} \ar@{-}[dll]_{\text{el. ab.}} \ar@{-}[dd] \\
L_{\fp} \ar@{-}[dd] &&&& M'_{1,\fp} \ar@{-}[dll]_{\text{el. ab.}} \\
&& N'_{\fp} \ar@{-}[dll]_{\text{el. ab.}} \\
L_{\fp}^{nr} \ar@{-}[d] \\
k_{\fp}
}
\end{equation*}
\caption{The tower of fields in Step 2. \label{figstep2}}
\end{figure}

Since $H=\Gal(M_{1,\fp}'/L^{nr}_{\fp})$ is abelian, the upper ramification group $H^{n_{\fq}-1}$ at $\fq$ has $p$-rank at most $e\cdot f_{\fq}$ by \cite[Prop. IV.7]{Ser79} (here $e$ is defined as in \ref{thm626}, and $f_{\fq}$ is the degree of $\fq$). There exists a subgroup $H'$ of $G_1=\Gal(M_1/L)$ which restricts isomorphically to $H^{n_{\fq}-1}$. Since $G_1$ is central in $\Gal(M_1/k)$, the subgroup $H'$ is normal in $\Gal(M_1/k)$. By Theorem \ref{thm626}, the fixed field under $H'$ is an extension $M_2$ of $L$ with $p$-rank
\begin{align}
&d_p(\Gal(M_2/L) \nonumber \\
=~&d_p(\Gal(N/L))+d_p G_{(S,\{n_{\fp}\}_{\fp})}-d_p G_{(S,\{\{n_{\fp}\}_{\fp\neq\fq},n_{\fq}-1\})}-d_p(H') \nonumber\\
\geq~&d_p(\Gal(N/L))+1+\Delta_{(S,\{n_{\fp}\}_{\fp})}+e\sum_{\fp\in S} f_{\fp}(n_{\fp}-1-[(n_{\fp}-1)/p]) \nonumber\\
&\qquad -1-\Delta_{(S,\{\{n_{\fp}\}_{\fp\neq\fq},n_{\fq}-1\})}-e\sum_{\fp\neq\fq} f_{\fp}(n_{\fp}-1-[(n_{\fp}-1)/p]) \nonumber\\
&\qquad\qquad -e\cdot f_{\fq}((n_{\fq}-1)+1+[(n_{\fq}-2)/p])-e\cdot f_{\fq} \label{dpstep2a}\\
=~&d_p(\Gal(N/L))+\Delta_{(S,\{n_{\fp}\}_{\fp})}-\Delta_{(S,\{\{n_{\fp}\}_{\fp\neq\fq},n_{\fq}-1\})} \nonumber\\
&\qquad +[(n_{\fq}-2)/p]-[(n_{\fq}-1)/p] \nonumber\\
=~&d_p(\Gal(N/L))+\Delta_{(S,\{n_{\fp}\}_{\fp})}-\Delta_{(S,\{\{n_{\fp}\}_{\fp\neq\fq},n_{\fq}-1\})}. \nonumber
\end{align}
The reason for the last step is as follows. As $N_{\fp}'/L_{\fp}$ is elementary abelian, Lemma \ref{lem518} shows that $p\nmid (n_{\fp}-1)$. Hence $[(n_{\fq}-1)/p]=[(n_{\fq}-2)/p]$.

By our construction, $M_2$ has a lower ramification depth than $N$ at $\fq$, and the ramification depths at other $\fp\neq\fq$ in $M_2$ is at most that of $N$. Repeat the process from the beginning of step 2 with $M_2$ in place of $N$, until all ramification of $N/L$ exceeding the depth $\nu$ has been removed. Call the resulting extension $N_1$.

Using the filtration
\begin{equation*}
\Delta_S\subseteq\cdots\subseteq\Delta_{S,\{n_{\fp}\}_{\fp}}\subseteq\Delta_{(S,\{\{n_{\fp}\}_{\fp\neq\fq},n_{\fq}-1\})}\cdots\subseteq\Delta_{S_k,\nu},
\end{equation*}
we get
\begin{equation}\label{dpstep2b}
d_p(\Gal(N_1/L))\geq d_p(\Gal(N/L))-(d_p(\Delta_{S_k,\nu})-d_p(\Delta_{S})).
\end{equation}
Finally, take the compositum $M$ of $N_1$ with extensions of $K$ ramified of depth bounded by $\nu$ and not already contained in $L/k$. The depth of the ramification in $M$ is the same as that of $N_1$, and
\begin{equation}\label{dpstep2c}
d_p(\Gal(M/L))\geq d_p(\Gal(N_1/L))+d_p(G_{S,\nu})-d_p(G).
\end{equation}

Combining \eqref{dpstep2a}, \eqref{dpstep2b}, \eqref{dpstep2c} above and \eqref{dpstep1} in step $1$, we get
\begin{multline}\label{dpstep2}
d_p(\Gal(M/L))\geq r_p(G)-d_p(G)-\sum_{\fp\in S\backslash S_{nr}}r_p(G_{\fp})-\sum_{\fp\in S_{nr}}(r_p(G_{\fp})-1) \\
+d_p(G_{S,\nu})+d_p\Delta_{S}-d_p\Delta_{S,\nu}.
\end{multline}
Note that by our construction, the extension $M/L$ is central over $L/k$. \\

\textbf{(Step 3: Remove the remaining ramification above $L$.)} Let $L_{\fp}^{ab}$ be the maximal abelian subextension of $L_{\fp}/k_{\fp}$, and let $G_{\fp}^{ab}$ be its Galois group. Let $I_{\fp}$ and $I_{\fp}^{ab}$ be the inertia group at $\fp$ of $L_{\fp}/k_{\fp}$ and $L_{\fp}^{ab}/k_{\fp}$ respectively. Let $I_{\fp}^{(p)}=I_{\fp}^{ab}/(I_{\fp}^{ab})^p$ be the maximal elementary abelian quotient of $I_{\fp}^{ab}$. Since $M/L$ is central over $L/k$, we have $M_{\fp}/L_{\fp}=M'_{\fp}L_{\fp}/L_{\fp}$ with $M_{\fp}'/k_{\fp}$ abelian. Let $L_{\fp}^{nr}$ be the maximal unramified subextension of $L_{\fp}^{ab}/k_{\fp}$, and let $L_{\fp}^{(p)}$ be the extension of $L_{\fp}^{nr}$ corresponding to $I_{\fp}^{(p)}$. Then $M'_{\fp}L_{\fp}^{(p)}/L_{\fp}^{nr}$ is elementary abelian. Since this is Galois, we can take the fixed field $M_{\fp}^{'r}$ under a complement of the inertia group of $M'_{\fp}L_{\fp}^{(p)}/L_{\fp}^{nr}$. The tower of fields is shown in Figure~\ref{figstep3}.

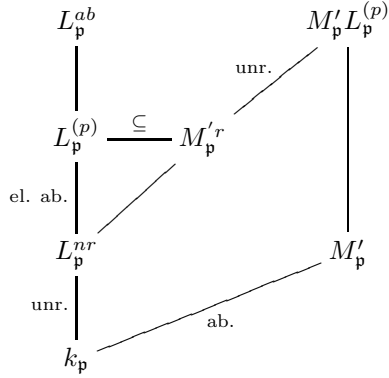
\begin{figure}[H]
\begin{equation*}
\xymatrix{
L_{\fp}^{ab} \ar@{-}[d] && M'_{\fp}L_{\fp}^{(p)} \ar@{-}[dl]_{\text{unr.}} \ar@{-}[dd] \\
L_{\fp}^{(p)} \ar@{-}[d]_{\text{el. ab.}} \ar@{-}[r]^{\subseteq} & M_{\fp}^{'r} \ar@{-}[dl] \\
L_{\fp}^{nr} \ar@{-}[d]_{\text{unr.}} && M_{\fp}' \ar@{-}[dll]^{\text{ab.}} \\
k_{\fp}
}
\end{equation*}
\caption{The tower of fields in Step 3. \label{figstep3}}
\end{figure}

The extension $M_{\fp}^{'r}/L_{\fp}^{nr}$ is totally ramified and elementary abelian. By Lemma \ref{lem523}, we have
\begin{equation} \label{dpstep3a}
d_p(\Gal(M_{\fp}^{'r}/L_{\fp}^{nr})) \leq e\cdot\deg{\fp}\cdot(\nu_{\fp}-1-[(\nu_{\fp}-1)/p]).
\end{equation}

Next, let $N'$ be the field extension of $L$ which remains after removing the ramification above $K$ at all $\fp$ (by taking the fixed fields of the preimages of $\Gal(M^{'r}_{\fp}/L_{\fp}^{(p)})$ for all $p$). We can estimate the drop in global $p$-rank using the formula
\begin{equation} \label{dpstep3b}
d_p(\Gal(M_{\fp}^{'r}/L_{\fp}{(p)}))=d_p(\Gal(M_{\fp}^{'r}/L_{\fp}^{nr}))-d_p(I_{\fp}^{(p)}),
\end{equation}
and
\begin{align} \label{dpstep3c}
d_p(I_{\fp}^{(p)})=d_p(I_{\fp}^{ab}/(I_{\fp}^{ab})^p) &= d_p(I_{\fp}^{ab}) \nonumber \\
&= \begin{cases}
d_p(G_{\fp})-1 &, \fp\in S_{nr}, \\
d_p(G_{\fp}) &, \fp\in S\backslash S_{nr}.
\end{cases}
\end{align}
The last formula is by Remark \ref{rem632}. We have
\begin{align*}
&d_p(\Gal(N'/L)) \\  \geq~ &d_p(\Gal(M/L))-\sum_{\fp\in S} d_p(\Gal(M^{'r}_{\fp}/L_{\fp}^{(p)})) \\
=~ &d_p(\Gal(M/L))-\sum_{\fp\in S} d_p(\Gal(M_{\fp}^{'r}/L_{\fp}^{nr}))-d_p(I_{\fp}^{(p)}) \text{\quad (by \eqref{dpstep3b})} \\
\geq~&d_p(\Gal(M/L))-e\cdot\sum_{\fp\in S} \deg{\fp}\cdot(\nu_{\fp}-1-[(\nu_{\fp}-1)/p]) \\
&\qquad +\sum_{\fp\in S}d_p(I_{\fp}^{(p)}) \text{\quad (by \eqref{dpstep3a})} \\
=~&d_p(\Gal(M/L))-(d_p(G_{S,\nu})-1-d_p(\Delta_{S,\nu}))+\sum_{\fp\in S}d_p(I_{\fp}^{(p)}) \\
\geq~&r_p(G)-d_p(G)-\sum_{\fp\in S\backslash S_{nr}}r_p(G_{\fp})-\sum_{\fp\in S_{nr}}(r_p(G_{\fp})-1) \\
& \qquad +d_p\Delta_{S}+1+\sum_{\fp\in S}d_p(I_{\fp}^{(p)}) \text{\quad (by \eqref{dpstep2})} \\
=~&r_p(G)-d_p(G)-\sum_{\fp\in S\backslash S_{nr}}(r_p(G_{\fp})-d_p(G_{\fp})) \\
& \qquad -\sum_{\fp\in S_{nr}}(r_p(G_{\fp})-d_p(G_{\fp}))+d_p\Delta_{S}+1 \text{\quad (by \eqref{dpstep3c})}\\
=~&r_p(G)-d_p(G)-\sum_{\fp\in S}(r_p(G_{\fp})-d_p(G_{\fp}))+d_p\Delta_{S}+1.
\end{align*}

Finally, to ensure that the primes in $T_k$ split completely, we replace $N'$ by the fixed field $\tilde{L}$ of the Frobenius of the primes in $T_k$. We have $d_p(\Gal(\tilde{L}/L)\geq d_p(\Gal(N'/L))-\abs{T_k}$. The theorem follows.
\end{proof}

\section{New lower bounds for the Ihara constants $A(2)$ and $A(3)$} \label{secA23}

\subsection{Preliminaries}

Before building our towers and proving Theorem \ref{thmtower}, we need some preliminary results on class field theory. The first result we need is an estimation of $r_p(G_{\fp})-d_p(G_{\fp})$.

\begin{prop} \label{propr-d}
Let $L/K$ be an abelian $p$-extension of global function fields over $\mathbb{F}_q$ of characteristic $p$, and let $\fp$ be a prime of degree $f_{\fp}$ in $K$. Assume that $L/K$ is ramified of depth at most $\nu_{\fp}$. Then
\begin{equation*}
r_p(G_{\fp})-d_p(G_{\fp}) \leq \binom{ef_{\fp}(\nu_{\fp}-1)+1}{2} = \frac{(ef_{\fp}(\nu_{\fp}-1))(ef_{\fp}(\nu_{\fp}-1)+1)}{2}.
\end{equation*}
Here $e$ is defined by $q=p^e$.
\end{prop}
\begin{proof}
Using the exact sequence
\begin{equation*}
\xymatrix@1{
1 \ar@{->}[r] & I_{\fp} \ar@{->}[r] & G_{\fp} \ar@{->}[r] & \mathbb{Z}/p^{f_{\fp}}\mathbb{Z} \ar@{->}[r] & 1,
}
\end{equation*}
we easily get $r_p(G_{\fp})-d_p(G_{\fp}) \leq r_p(I_{\fp})$. By \cite[Prop. IV.7]{Ser79}, $I_{\fp}$ is a $p$-group of order at most $p^{e\cdot f_{\fp}(\nu_{\fp}-1)}$. Hence
\begin{equation*}
r_p(I_{\fp}) \leq \binom{ef_{\fp}(\nu_{\fp}-1)+1}{2}.
\end{equation*}
This completes the proof.
\end{proof}

Next, we turn our attention to ray class fields. Let $K$ be a global field, $T$ be a finite set of primes in $K$ and $\fm=\sum_{\fp}m_{\fp}\fp$ be a ray modulus with support disjoint from $T$. Let $K_{\fm}^T$ be the $T$-ray class field of conductor $\fm$ in $K$, i.e. the maximal subfield of the ray class field $K_{\fm}$ of conductor $\fm$ such that all primes in $T$ split completely.
To compute the genus of the ray class fields, we will use the ``F\"{u}hrerproduktdiskriminantformel'' (see \cite[Chapter 5]{Wei48}).
\begin{prop}[F\"{u}hrerproduktdiskriminantformel] \label{propcdf}
Let $K/F$ be a geometric extension of function fields over $\mathbb{F}_q$ with an abelian Galois group $G$, then
\begin{equation*}
2g(K)-2=[K:F](2g(F)-2)+\sum_{\chi}\deg\ff_\chi.
\end{equation*}
Here $\ff_\chi$ is the Artin-conductor of $\chi$, and the sum runs through the characters $\chi$ of $G$.
\end{prop}

The next proposition gives a lower bound for the $p$-rank of the ray class group $\Gal(k_{\fm}^T / k)$.
\begin{prop} \label{propdpK}
The $p$-rank of the extension $k_{\fm}^T / k$ is at least
\[
d_p(\Gal(k_{\fm}^T / k)) \geq 1+\sum_{\fp\in S} ef_{\fp} \cdot (\nu_{\fp}-1-[(\nu_{\fp}-1)/p])-t.
\]
\end{prop}
\begin{proof}
For a singleton $T$, the group contains the factor $\prod_{\fp \in S} U_{\fp}^{(1)} / U_{\fp}^{(\nu_{\fp})}$. Now use Proposition \ref{prop627}
and subtract $t-1$ for a set $T$ of size $t$.
\end{proof}

\subsection{Construction of the tower and the proof of Theorem \ref{thmtower}}

Let $k$ be a function field over $\mathbb{F}_q$, called the \textit{base field}. Let $g$ and $N$ be the genus and the number of rational places of $k$ respectively. Let $S$ be a set of places which we allow to ramify, and let $T$ be a set of degree one places disjoint from $S$, of cardinalities $s=\abs{S}$ and $t=\abs{T}$. Clearly we have $t\leq N$. Let $\fm=\sum_{\fp\in S}\nu_{\fp}\fp$, where $\nu_{\fp}\in\mathbb{N}$ for each $\fp$. For any prime $\fp$ of $k$, denote by $f_{\fp}$ its degree. Let $K'$ be the ray class field of conductor $\fm$, and let $K=k_{\fm}^T$ be the maximal subfield of $K'$ such that all places in $T$ splits completely.

Let $T_K$ be the primes above $T$ in $K$. Now we build the $(T_K,p)$-class field tower on top of $K$, and let $L$ be the union of the tower. Let $G=\Gal(L/k)$. The extension of fields is shown in the following diagram.

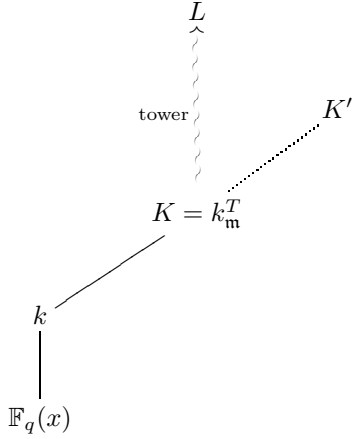
\begin{figure}[H]
\begin{equation*}
\xymatrix{
& L \ar@{<~}[dd]_{\text{tower}} \\
&& K' \ar@{.}[dl] \\
& K=k_{\fm}^T \ar@{-}[dl] \\
k \ar@{-}[d] \\
\mathbb{F}_q(x)
}
\end{equation*}
\caption{The tower construction \label{figtower}}
\end{figure}

We determine the conditions to be met in order for the tower to be infinite. Suppose that $G$ is finite, nontrivial and let $d=d_p(G)$, $r=r_p(G)$. From Proposition \ref{propdpK},
we have
\begin{equation}\label{eqndp}
d \geq 1+\sum_{\fp\in S}ef_{\fp}(\nu_{\fp}-1-[(\nu_{\fp}-1)/p])-t.
\end{equation}
From Theorem \ref{thm810} and Proposition \ref{propr-d}, we have
\begin{equation}\label{eqnr-d}
r-d \leq \sum_{\fp\in S}\frac{(ef_{\fp}(\nu_{\fp}-1))(ef_{\fp}(\nu_{\fp}-1)+1)}{2} +(t-1).
\end{equation}
Now $G$ is infinite if \eqref{eqndp} and \eqref{eqnr-d} together yield a contradiction
in the Golod-Shafarevich inequality
\begin{equation*}
r-d > \frac{d^2}{4}-d. 
\end{equation*}
This happens when
\begin{multline*}
\frac{(1+\sum_{\fp\in S}ef_{\fp}(\nu_{\fp}-1-[(\nu_{\fp}-1)/p])-t)^2}{4}-(1+ \sum_{\fp\in S}ef_{\fp}(\nu_{\fp}-1-[(\nu_{\fp}-1)/p])-t) \\
\geq \frac{(\sum_{\fp\in S}ef_{\fp}(\nu_{\fp}-1))(ef_{\fp}(\nu_{\fp}-1)+1)}{2} +(t-1),
\end{multline*}
which simplifies to
\begin{multline*}
(1+\sum_{\fp\in S}ef_{\fp}(\nu_{\fp}-1-[(\nu_{\fp}-1)/p])-t)^2 \\
- 2\sum_{\fp\in S}ef_{\fp}(\nu_{\fp}-1)(ef_{\fp}(\nu_{\fp}-1)+1) - 4\sum_{\fp\in S}ef_{\fp}(\nu_{\fp}-1-[(\nu_{\fp}-1)/p]) \geq 0.
\end{multline*}
This is the condition \eqref{ineqr-d} in Theorem \ref{thmtower}.

Now suppose our tower is infinite (by a suitable choice of $S$, $T$ and $\nu=(\nu_{\fp} : \fp \in S)$ so that \eqref{ineqr-d} is satisfied), then we can calculate the lower bound of $A(q)$ given by this tower as follows. By Proposition \ref{propcdf}, we have
\begin{equation*}
2g(K)-2=[K:k](2g-2)+\sum_{\chi}\deg\ff_\chi,
\end{equation*}
where $\chi$ runs through the characters of $\Gal(K/k)$. So,
\begin{equation}\label{eqng-1K}
g(K)-1=[K:k]\left(g-1+\frac{1}{2[K:k]}\sum_{\chi}\deg\ff_\chi\right).
\end{equation}

In $K$, the number of places that splits in the tower $L/K$ is $\abs{T_K}=[K:k]t$. Therefore, by \eqref{eqnT/g} and \eqref{eqng-1K}, we obtain the lower bound of $A(q)$ given by this tower.
\begin{equation*}
A(q) \geq \frac{t[K:k]}{g(K)-1} = \frac{t}{g-1+\frac{1}{2[K:k]}\sum_{\chi}\deg\ff_\chi}.
\end{equation*}

This completes the proof of Theorem \ref{thmtower}.


\begin{remark}
In \cite[Section 11]{Kuh02}, Corollary \ref{cortower} is obtained using a different argument, based on the ``Ray class fields \`{a} la Hayes'' (see
\cite[Example 1.5]{Aue00}), the tower in Figure \ref{figtower} and the observation that
\begin{equation*}
\frac{N(K')}{g(K')-1} \leq \frac{N(K)}{g(K)-1} \leq \frac{N(k)}{g(k)-1}.
\end{equation*}
Therefore, Theorem \ref{thmtower} can be viewed as an improvement to Kuhnt's result.
\end{remark}


\subsection{New lower bounds for $A(2)$ and $A(3)$}


With Theorem \ref{thmtower} in hand, it remains for us to find a function field $k$ so that the theorem is applicable. For this we look for
function fields with many rational places with respect to their genus and with sufficiently many other places of small degree.
For $q=2$, we construct two infinite towers. Their asymptotic limits are as follows.

\begin{prop} \label{propF21}
Let $E=\mathbb{F}_2(x,y)$ for $y^2+y=x^3+x$.
For each $n \geq 0$, there exists a function field of degree $2^n$ over $E$ with $N=5 \cdot 2^n$ rational places and with genus $g$
such that $N/g \geq 0.316837.$
\end{prop}

\begin{prop} \label{propF22}
For each $n \geq 0$, there exists a function field of degree $2^n$ over $\mathbb{F}_2(x)$ with $N=3 \cdot 2^n$ rational places and with genus $g$
such that $N/g \geq 0.316999.$
\end{prop}
The asymptotic limit of the second tower is our lower bound in Theorem \ref{thmA2}.

For the first tower we start with the function field $E=\mathbb{F}_2(x,y)$ of the elliptic curve $y^2+y=x^3+x$.
Denoting by $a_d$ the number of places of degree $d$, we have
\[
(a_d(E) : d \geq 1 ) = (5,0,0,5,4,10,20,25,\ldots), \qquad g(E)=1.
\]
Let $P_4$ and $P_5$ be places of $E$ of degree $4$ and $5$ respectively, and let $E'$ be the ray class field of conductor $2P_4+2P_5$
in which all $5$ rational places of $E$ split completely. By Proposition \ref{propdpK}, we have $d_2(\Gal(E'/E))\geq 1+4+5-5=5$.
Thus there is a subfield $k$ of $E'$ so that all the $5$ rational places of $E$ split completely and $\Gal(k/E)$ is an elementary abelian group of order $32$.
In particular $a_1(k)=32 \cdot 5=160$. To calculate the genus of $k$, we use Proposition \ref{propcdf}. One can show that there is no proper extension of $E$ with conductor $2P_4$ so that all $5$ rational places split. In $k/E$,
there is a unique degree $2$ subextension of conductor $2P_5$.
The characters for the remaining degree $2$ subextensions have conductor $2P_4+2P_5$.
Hence $2g(k)-2=32(0)+1\cdot10+30\cdot18,$ and $g(k)=276$.
To apply Theorem \ref{thmtower}, we need a suitable set $S$ of places to ramify. For this we analyse the places of small degree.
There is a unique place of degree $5$ in $k$ above $P_5$, which is fully ramified in $k/E$. Notice that the extension $k/E$ is elementary abelian and therefore it is a compositum of degree $2$ Artin-Schreier extensions (see \cite[Appendix A.13]{Sti09}). An explicit model for $k$ is given by the compositum of the extensions $E(v)$ with
\begin{equation*}
v^2+((x^2+x)(xy+x+y) + 1)v=(x^2+x)h,
\end{equation*}
where $h$ is in the span of the functions $\{1,x,y,x^2,x^3\}$. The unique degree $2$ subextension $C/E$ with conductor $2P_5$ corresponds to $h=x$.
For the extensions $C/E$ and $k/E$ we have
\begin{align*}
(a_d(C) : d \geq 1 ) &= (10,0,0,0,3,\ldots), \quad g(C) = 6. \\
(a_d(k) : d \geq 1 ) &= (160,0,0,0,1,0,0,65,0,48,\ldots), \quad g(k) = 276.
\end{align*}
The five places of degree 4 in $E$ are inert in $C/E$. The place of degree $8$ above $P_4$ ramifies completely in $k/C$ and the places above the other places
of degree $4$ split completely in $k/C$ (because $k/E$ is elementary abelian), giving a total of $65$ degree $8$ places for $k$. The two nonramified places of
degree $5$ in $C$ each decompose into $8$ places of degree $10$ in $k$.
Now let $S$ consist of one degree $5$ place, $27$ degree $8$ places, and one degree $10$ place, let $\nu_{\fp}=2$ for all $\fp\in S$, and form the conductor
$\fm=\sum_{\fp\in S}2\fp$.
Then one can check easily that the inequality \eqref{ineqr-d} is satisfied for $t = |T| = 160$ and the class field tower of $K=k_\fm^T$ is
infinite. With Corollary \ref{cortower} we find
\[
A(2) \geq \frac{160}{276-1+\frac{1}{2} \cdot 2 \cdot (1 \cdot 5 + 27 \cdot 8 + 1 \cdot 10)} = 80/253 = 0.316205\ldots. 
\]
The place of degree $5$ contributes to a fraction of at most $31/32$ of the characters for $K/k$, a place of degree $8$ to a fraction of at most $255/256$, and
the place of degree $10$ to a fraction of at most $1023/1024$.
Using this as an upper bound for the average conductor, Theorem \ref{thmtower} yields
\begin{align*}
A(2) &\geq \frac{160}{276-1+\frac{1}{2} \cdot 2 \cdot (1 \cdot 5 \, (1-2^{-5}) + 27 \cdot 8 \, (1-2^{-8}) + 1 \cdot 10 \, (1-2^{-10}))} \\
     &= \frac{2^{14}}{2^9 \cdot 101 -1} = 0.316837 > 32/101.
\end{align*}
We have shown Proposition \ref{propF21}.

We construct the second tower to prove Theorem \ref{thmA2}.
Let $H$ be the degree two extension of the rational function field with conductor $2P_3$, $P_3$ a place of degree $3$, so that all $3$ rational places split.
A model for $H=\mathbb{F}_2(x,y)$ is given by $y^2+(x^3+x+1)y = x^2+x$. We have
\[
(a_d(H) : d \geq 1 ) = (6,0,1,1,6,\ldots), \quad g(H) = 2.
\]
For two places $P_5$ and $P'_5$ of degree $5$, let $k/H$ be an elementary abelian extension of degree $32$ with conductor $2P_5+2P'_5$ so that all $6$ rational places
split completely. Thus $a_1(k) = 32 \cdot 6 = 192$, and $2g(k)-2=32(2)+31\cdot20$ shows that $g(k)=343.$ An explicit model for $k$ is given by the compositum of the extensions $E(v)$ with
\begin{equation*}
v^2+(x^5+x^2+1)v=(x^2+x)h,
\end{equation*}
where $h$ is in the span of the functions $\{1,x,x^2,y,y^2\}$. We have
\[
(a_d(k): d \geq 1 ) = (192, 0, 0, 0, 2, 16, 0, 16, 0, 64, \ldots), \quad g(k)=343.
\]
Let the set $S$ consist of $2$ places of degree $5$, $16$ places of degree $6$, $15$ places of degree $8$, and $4$ places of degree $10$.
For $\fm = \sum_{\fp \in S} 2 \fp$, and for $|T|=192$, the field $K = k_\fm^T$ has an infinite class field tower and
\[
A(2) \geq \frac{192}{343-1+\frac{1}{2} \cdot 2 \cdot (2 \cdot 5 + 16 \cdot 6 + 15 \cdot 8 + 4 \cdot 10)} = 6/19 = 0.315789\ldots.
\]
As before, using
\[
f' = 2 \cdot (2 \cdot 5 \, (1-2^{-5}) + 16 \cdot 6 \, (1-2^{-6}) + 15 \cdot 8 \, (1-2^{-8}) + 4 \cdot 10 \, (1-2^{-10}))
\]
as an upper bound for the average conductor of $K/k$, Theorem \ref{thmtower} yields
\[
A(2) \geq \frac{192}{343-1+\frac{1}{2} f'} \geq 0.316999\ldots.
\]
We have shown Proposition \ref{propF22}.

Now we turn our attention to $q=3$. The tower we construct has the following asymptotic limit.

\begin{prop}\label{propF31}
Let $E=\mathbb{F}_3(x,y)$ for $y^2=x^3-x+1$. For each $n \geq 0$,
there exists a function field of degree $3^n$ over $E$ with $N=7 \cdot 3^n$ rational places and with genus $g$
such that $N/g \geq 0.492876.$
\end{prop}

Again we consider the function field $E$ of a maximal elliptic curve. This time we take $E=\mathbb{F}_3(x,y)$ with $y^2=x^3-x+1$. We have $g(E)=1$ and $(a_i(E) : i \geq 1) = (7,0,7,21,42,\ldots).$
Let $P_5$ be one of the degree $5$ places of $E$ and let $E'$ to be the ray class field of $E$ of conductor $3P_5$ so that all $7$ rational places
of $E$ split completely. By Proposition \ref{propdpK}, we have $d_3(\Gal(E'/E))\geq 1+2\cdot5-7=4$. Thus there is a subfield $k$ of $E'$
so that all the $7$ rational places of $E$ split completely and $k/E$ is elementary abelian of order $81$.
In particular $[k:E]=81$ and $a_1(k)=81\cdot7=567$. Using Proposition \ref{propcdf}, we have $2g(k)-2=3^4(0)+80\cdot3\cdot5,$
so that $g(k)=601$. An explicit model of $k$ is given by the compositum of the extensions $E(v)$ with
\begin{equation*}
v^3-(xy+x^2-1)^2v=(x^3-x)h,
\end{equation*}
where $h$ is in the span of the functions $\{1,x,y,xy\}$. To see the splitting of the finite rational places we note that
$(xy+x^2-1)^2 = (x^3-x)(x^2-y+x)+1.$ For the extension $k$ we find
\[
(a_d(k) : d \geq 1 ) = (567, 0, 0, 0, 1, 0, 0, 162, 1809), \quad  g(k)=601.
\]
If we let $S$ be a set of $46$ places of degree $8$ then, for $\fm = \sum_{\fp \in S} 3 \fp$ and $|T|=567$, the class field tower of $K=k_{\fm}^T$
is infinite with
\[
A(3) \geq \frac{567}{601-1+\frac{1}{2} \cdot 3 \cdot (46\cdot8)} = \frac{63}{128} = 0.4921875.
\]
The same construction with $S$ a set of one place of degree $5$, $43$ places of degree $8$, and two places of degree $9$, yields an infinite
class field tower with $|T|=567$. Using Theorem \ref{thmtower},
\begin{align*}
A(3) & \geq \frac{567}{601-1+\frac{1}{2} \cdot 3 \cdot (1\cdot5\,(1-3^{-5})+43\cdot8\,(1-3^{-8})+2\cdot9\,(1-3^{-9}))} \\
     & = 0.492876\ldots.
\end{align*}

This proves Proposition \ref{propF31} and Theorem \ref{thmA3}.

\begin{remark}

In this remark, we compare our method with the usual method of constructing class field towers.

Let $L/K/k$ be a tower of $p$-extensions, where $L/K$ is the usual unramified $(T,p)$-class field tower of Serre in which the places in
$T$ split completely and $K/k$ is a finite Galois $p$-extension with some (wild) ramifications. The construction of class field towers in \cite{NiXi98} and \cite{XiYe07} are of this type with $K/k$ elementary abelian of rank $l$. We illustrate how the new inequalities apply to the known towers in \cite{NiXi98} and \cite{XiYe07}.

We consider the special case that the ramification in $K/k$ is concentrated at a set $S$ consisting of $s$ rational
places that ramify completely and the set $T$ consists of all places of $K$ above $t$ completely split rational places of $k$ and possibly $s'$
ramified places. The restriction $T_k$ of $T$ to $k$ is therefore of size $t+s'$. For a finite tower $L/K$, the generator rank and the relation rank
of the group $\Gal(L/K)$ satisfy the bounds
\begin{equation}\label{eqn41}
\text{(usual method)}~~
\begin{cases}  d_p \geq s l - (|T_k| -1) - l  \\ r_p -d_p \leq |T|-1.
\end{cases}
\end{equation}
In \cite{Kuh02}, Kuhnt considers the Galois group $\Gal(L/k)$ instead of the usually considered $\Gal(L/K)$. The bound for $d_p$
changes by $l$. For $r_p - d_p$, there are two cases. For $\fp\in S$, we have
\[
r_p(G_{\fp})-d_p(G_{\fp}) \leq r_p(I_{\fp}) \leq \binom{\ell+1}{2}.
\]
For $\fp \in S \cap T_k$, we have $d_p(G_p) \leq \ell$, and hence
\[
r_p(G_{\fp})-d_p(G_{\fp}) \leq \binom{d_p(G_{\fp})+1}{2}-d_p(G_{\fp}) \leq  \binom{\ell+1}{2} - \ell.
\]
Thus by Theorem \ref{thm810}, we have
\[
r_p-d_p \leq s\binom{\ell+1}{2}-s' \ell+\abs{T_k}-1.
\]
Therefore, the inequalities in our method are
\begin{equation} \label{eqn42}
\text{(our method)}~~
\begin{cases}  d_p \geq s l - (|T_k| -1)   \\ r_p -d_p \leq s \binom{l+1}{2} - s' l + |T_k|-1,
\end{cases}
\end{equation}
if the tower is finite.

Now we have two sets of inequalities \eqref{eqn41} and \eqref{eqn42}, that each may be combined with the Golod-Shafarevich inequality for a proof that $L/K$ is infinite. The constructions in \cite{NiXi98}, \cite{XiYe07} have $\Gal(K/k)=(\mathbb{Z}/2\mathbb{Z})^2$ with inequalities
\begin{align*}
\text{(usual method)}~&d_p \geq 20,~r_p-d_p \leq 80~~(s=21, t=20, s'=1, |T|=81), \\
\text{(our method)}~&d_p \geq 21,~r_p-d_p \leq 82~~(s=21, t=21, s'=1, |T|=85),
\end{align*}
and
\begin{align*}
\text{(usual method)}~&d_p \geq 22,~r_p-d_p \leq 96~~(s=24, t=24, s'=1, |T|=97), \\
\text{(our method)}~&d_p \geq 22,~r_p-d_p \leq 92~~(s=24, t=24, s'=3, |T|=99),
\end{align*}
respectively. In these cases, the inequalities
in our method show that the original towers contain infinite subtowers with more completely split points.
\end{remark}

\subsection*{Acknowledgments}
We thank an anonymous referee for carefully reading our manuscript and for
making several suggestions that improved the presentation.

\end{document}